\documentclass{amsart}

\usepackage{amsmath,amssymb,latexsym,amsthm,newlfont,enumerate}
\usepackage[all]{xy}

\theoremstyle{plain}

\newtheorem{thm}{Theorem}[section]

\newtheorem{pro}[thm]{Proposition}

\newtheorem{lem}[thm]{Lemma}

\theoremstyle{definition}

\newtheorem{dfn}[thm]{Definition}

\newtheorem{rem}[thm]{Remark}

\theoremstyle{remark}

\newcommand{\Z}{\mathbb{Z}}
\newcommand{\N}{\mathbb{N}}
\newcommand{\C}{\mathbb{C}}
\newcommand{\R}{\mathbb{R}}
\newcommand{\Q}{\mathbb{Q}}
\newcommand{\PS}{\mathbb{P}}

\newcommand{\OO}{\mathcal{O}}

\newcommand{\can}{\mathrm{can}}

\newcommand{\lto}{\longrightarrow}

\DeclareMathOperator{\mult}{mult}

\DeclareMathOperator{\Supp}{Supp}

\DeclareMathOperator{\Spec}{Spec}

\DeclareMathOperator{\Nef}{Nef}

\DeclareMathOperator{\Effb}{\overline{\mathrm{Eff}}}

\DeclareMathOperator{\trdeg}{trdeg}

\title{A note on the abundance conjecture}

\author{Tobias Dorsch}
\address{Mathematisches Institut, Universit\"at Bonn, Endenicher Allee 60, 53115 Bonn, Germany}
\email{dorsch@math.uni-bonn.de}

\author{Vladimir Lazi\'c}
\address{Mathematisches Institut, Universit\"at Bonn, Endenicher Allee 60, 53115 Bonn, Germany}
\email{lazic@math.uni-bonn.de}

\thanks{We were supported by the DFG-Emmy-Noether-Nachwuchsgruppe ``Gute Strukturen in der h\"oherdimensionalen birationalen Geometrie". We would like to thank C.~Xu for helpful discussions and comments about toroidal embeddings, cf.~Remark \ref{rem:toroidal}, and J.~Koll\'ar, S.~Kov\'acs, M.~Musta\c{t}\u{a} and Th.~Peternell for useful conversations related to this work.}

\begin{document}

\begin{abstract}
We prove that the abundance conjecture for non-uniruled klt pairs in dimension $n$ implies the abundance conjecture for uniruled klt pairs in dimension $n$, assuming the Minimal Model Program in lower dimensions.
\end{abstract}

\maketitle

\setcounter{tocdepth}{1}
\tableofcontents

\section{Introduction}

The main outstanding conjecture in the Minimal Model Program for projective varieties in characteristic zero is that every klt pair $(X,\Delta)$ with $K_X+\Delta$ pseudoeffective has a minimal model $(Y,\Delta_Y)$ such that $K_Y+\Delta_Y$ is semiample. Such a minimal model is called a \emph{good model}. It is well known that the existence of good models implies the abundance conjecture, which predicts that every minimal model is good.

We say that a pair is uniruled if the underlying variety is so, and similarly for a non-uniruled pair. In this paper, we show that it suffices to prove the aforementioned conjectures for non-uniruled pairs. More precisely, the following are our main results.

\begin{thm}\label{thm:main2}
Assume the existence of good models for klt pairs in dimensions at most $n-1$. 

If the abundance conjecture holds for non-uniruled klt pairs in dimension $n$, then the abundance conjecture holds for uniruled klt pairs in dimension $n$.
\end{thm}

\begin{thm}\label{thm:main3}
Assume the existence of good models for klt pairs in dimensions at most $n-1$. 

Then the existence of good models for non-uniruled klt pairs in dimension $n$ implies the existence of good models for uniruled klt pairs in dimension $n$.
\end{thm}

Observe that by passing to a terminal model, cf.~Theorem \ref{thm:dltblowup}, and by using the main result of \cite{BDPP}, Theorems \ref{thm:main2} and \ref{thm:main3} show that it suffices to prove the existence of good models and the abundance conjecture for terminal pairs $(X,\Delta)$ with $K_X$ pseudoeffective.

The existence of good models for surfaces is classical. For terminal threefolds, minimal models were constructed in \cite{Mor88,Sho85}, whereas minimal models of canonical fourfolds exist by \cite{BCHM,Fuj05}. In higher dimensions, the existence of minimal models for klt pairs of log general type is proved in \cite{HM10,BCHM}, and by different methods in \cite{CL12a,CL13}, whereas abundance holds for such pairs by \cite{Sho85,Kaw85b}. Minimal models for effective klt pairs exist assuming the Minimal Model Program in lower dimensions \cite{Bir11}. 

The abundance conjecture was proved in \cite{Miy87,Miy88a,Miy88b,Kaw92} for terminal threefolds, and extended to log canonical threefold pairs $(X,\Delta)$ in \cite{KMM94}. The proof in \cite{KMM94} proceeds by running a $K_X$-MMP with scaling of $\Delta$ which is $(K_X+\Delta)$-trivial, to end up either with a Mori fibre space, or with a model $(Y,\Delta_Y)$ on which $K_Y+(1-\varepsilon)\Delta_Y$ is nef for every $0\leq\varepsilon\ll1$. In the first case one is almost immediately done by induction even in higher dimensions, whereas in the second case one uses Chern classes, the geometry of surfaces and the case by case analysis of the numerical Kodaira dimension -- the argument follows closely the proof for terminal threefolds. A variation of the first case was implemented in \cite{DHP13}, and we recall it in Theorem \ref{thm:tau} below. However, this does not cover all uniruled pairs, as we explain in Remark \ref{rem:DHP}. Here we take a different approach to reduce to the case of smooth varieties with effective canonical class.

We briefly explain the strategy of the proof. If $(X,\Delta)$ is a uniruled klt pair, then by \cite[Proposition 8.7]{DHP13} we may assume that the adjoint divisor $K_X+\Delta$ is effective. We first show that we may furthermore assume that $X$ is smooth and $\Delta$ is a reduced simple normal crossings divisor, and that there exists an effective $\Q$-divisor $D$ such that $K_X+\Delta\sim_\Q D$ and the supports of $\Delta$ and $D$ are the same. Then we use ramified covers, dlt models and log resolutions to construct a log smooth pair $(W,\Delta_W)$ and a generically finite morphism $w\colon W\to X$ such that $K_W$ is an effective divisor -- we do this by carefully analysing the behaviour of valuations under finite morphisms. We conclude by the construction of $w$ and since the Kodaira dimension and the numerical Kodaira dimension are preserved under proper morphisms, cf.\ Lemma \ref{lem:properpullbacks}.

In fact, our techniques lead to the following main technical result of the paper, which implies Theorems \ref{thm:main2} and \ref{thm:main3}.

\begin{thm}\label{thm:main}
Assume the existence of good models for klt pairs in dimensions at most $n-1$. 

If good models exist for log smooth klt pairs $(X,\Delta)$ of dimension $n$ such that the linear system $|K_X|$ is not empty, then good models exist for uniruled klt pairs in dimension $n$.
\end{thm}

As a by-product, we obtain in Lemma \ref{lem:IndexOne} a result which can be viewed as a global version of the index one cover \cite[Corollary 1.9]{Rei80}, and might be of independent interest.

\section{Notation and previous results}\label{sec:auxiliary}

In this section we gather previous results which will be used in Section \ref{sec:3}. We pay special attention to the behaviour of discrepancies under finite morphisms -- this is also known, but we provide the details for the benefit of the reader.

Throughout the paper we work over $\C$. A \emph{pair} $(X,\Delta)$ consists of a normal variety $X$ and a Weil $\Q$-divisor $\Delta\geq0$ such that the divisor $K_X+\Delta$ is $\Q$-Cartier. Such a pair is \emph{log smooth} if $X$ is smooth and if the support of $\Delta$ has simple normal crossings. We use extensively singularities of pairs, the standard reference is \cite{KM98}. Unless explicitly stated otherwise, all varieties are normal and projective.

\subsection{Terminal and dlt models}

Terminal and dlt models allow us to make the singularities of pairs simpler, in the first case by replacing klt by terminal singularities, and in the second case by replacing log canonical by dlt singularities. For us, particularly the dlt models and their precise definition will be useful.

\begin{dfn}
Let $(X,\Delta)$ be a klt pair. A pair $(Y,\Gamma)$ together with a proper birational morphism $f\colon Y \to X$ is a \emph{terminal model of $(X,\Delta)$} if the following holds:
\begin{enumerate}
\item[(i)] the pair $(Y,\Gamma)$ is terminal,
\item[(ii)] $Y$ is $\Q$-factorial, 
\item[(iii)] $K_Y+\Gamma\sim_\Q f^*(K_X+\Delta)$.
\end{enumerate}
\end{dfn}

\begin{dfn}
Let $(X,\Delta)$ be a log canonical pair. A pair $(Y,\Gamma)$ together with a proper birational morphism $f\colon Y\to X$ is a \emph{dlt model of $(X,\Delta)$} if the following holds:
\begin{enumerate}
\item[(i)] the pair $(Y,\Gamma)$ is dlt,
\item[(ii)] the divisor $\Gamma$ is the sum of $f^{-1}_*\Delta$ and all exceptional prime divisors,
\item[(iii)] $Y$ is $\Q$-factorial, 
\item[(iv)] $K_Y+\Gamma\sim_\Q f^*(K_X+\Delta)$.
\end{enumerate}
\end{dfn}

The starting point is the following existence result.

\begin{thm}\label{thm:dltblowup}
Let $(X,\Delta)$ be a pair. 
\begin{enumerate}
\item[(a)] If $(X,\Delta)$ is klt, then a terminal model of $(X,\Delta)$ exists.
\item[(b)] If $(X,\Delta)$ is log canonical, then a dlt model of $(X,\Delta)$ exists.
\end{enumerate}
\end{thm}

\begin{proof}
For part (a), see \cite[Corollary 1.4.3]{BCHM} and the paragraph after that result. Part (b) is \cite[Theorem 3.1]{KK10}.
\end{proof}

Recall that a variety $X$ of dimension $n$ is uniruled if there is a dominant rational map $\PS^1\times Y\dashrightarrow X$, for some variety $Y$ with $\dim Y=n-1$. This property is preserved in the birational equivalence class of $X$. The following result is fundamental.

\begin{thm}\label{thm:BDPP}
Let $X$ be a projective variety with canonical singularities. Then $X$ is uniruled if and only if $K_X$ is not pseudoeffective.
\end{thm}

\begin{proof}
For manifolds, this is \cite[Corollary 0.3]{BDPP}. The result for varieties with canonical singularities follows immediately.
\end{proof}

\subsection{Good models}

We recall the definition of log terminal and good models.

\begin{dfn}
Let $X$ and $Y$ be $\Q$-factorial varieties, and let $D$ be a $\Q$-divisor on $X$. A birational contraction $f\colon X\dashrightarrow Y$ is a \emph{log terminal model for $D$} if $f_*D$ is nef, and if there exists a resolution $(p,q)\colon W\to X\times Y$ of the map $f$ such that $p^*D=q^*f_*D+E$, where $E\geq0$ is a $q$-exceptional $\Q$-divisor which contains the whole $q$-exceptional locus in its support. If additionally $f_*D$ is semiample, the map $f$ is a \emph{good model} for $D$.
\end{dfn}

Note that if $(X,\Delta)$ is a klt pair, then it has a good model if and only if there exists a Minimal Model Program with scaling of an ample divisor which terminates with a good model of $(X,\Delta)$, cf.\  \cite[Propositions 2.4 and 2.5]{Lai11}.

\begin{thm}\label{thm:birkar}
Assume the existence of good models for klt pairs in dimensions at most $n-1$. 

Let $(X,\Delta)$ be a klt pair of dimension $n$ such that $\kappa(X,K_X+\Delta)\geq0$. Then $(X,\Delta)$ has a log terminal model.
\end{thm}

\begin{proof}
By \cite[Corollary 1.7 and the paragraph after Definition 2.2]{Bir11}, it is enough to show that every $\Q$-factorial dlt pair $(Y,\Gamma)$ of dimension at most $n-1$ such that $K_Y+\Gamma$ is pseudoeffective has a minimal model in the sense of Birkar and Shokurov, cf.\ \cite[Definition 2.1]{Bir11}. To this end, note first that $\kappa(Y,K_Y+\Gamma)\geq0$ by our assumption and by \cite[Theorem 1.5]{Gon11}. Then we conclude by induction and by \cite[Corollary 1.7]{Bir11} again.
\end{proof}

Kawamata \cite{Kaw85} was the first to realise that the numerical Kodaira dimension, in the case of nef divisors, plays a crucial role in the abundance conjecture. The concept was generalised in \cite{Nak04} to the case of pseudoeffective divisors.

\begin{dfn}
Let $X$ be a smooth projective variety and let $D$ be a pseudoeffective $\Q$-divisor on $X$. If we denote
$$\sigma(D,A)=\sup\big\{k\in\N\mid \liminf_{m\rightarrow\infty}h^0(X,\lfloor mD\rfloor+A)/m^k>0\big\}$$
for a Cartier divisor $A$ on $X$, then the {\em numerical Kodaira dimension\/} of $D$ is
$$\kappa_\sigma(X,D)=\sup\{\sigma(D,A)\mid A\textrm{ is ample}\}.$$
If $X$ is a projective variety and if $D$ is a pseudoeffective $\Q$-Cartier $\Q$-divisor on $X$, then we set $\kappa_\sigma(X,D)=\kappa_\sigma(Y,f^*D)$ for any birational morphism $f\colon Y\to X$ from a smooth projective variety $Y$.
\end{dfn}

The function $\kappa_\sigma$ behaves similarly to the Kodaira dimension under proper pullbacks:

\begin{lem}\label{lem:properpullbacks}
Let $D$ be a $\Q$-divisor on a $\Q$-factorial variety $X$, and let $f\colon Y\to X$ be a proper surjective morphism. Then 
$$\kappa(X,D)=\kappa(Y,f^*D)\quad\text{and}\quad\kappa_\sigma(X,D)=\kappa_\sigma(Y,f^*D).$$ 
If $f$ is birational and $E$ is an effective $f$-exceptional divisor on $Y$, then 
$$\kappa(X,D)=\kappa(Y,f^*D+E)\quad\text{and}\quad\kappa_\sigma(X,D)=\kappa_\sigma(Y,f^*D+E).$$ 
\end{lem}

\begin{proof}
The first three relations are \cite[Lemma II.3.11, Proposition V.2.7(4)]{Nak04}. For the last one, we have $P_\sigma(f^*D+E)=P_\sigma(f^*D)$ by \cite[Lemma 2.16]{GL13}, hence $\kappa_\sigma(Y,f^*D+E)=\kappa_\sigma(Y,f^*D)$ by \cite[Theorem 6.7]{Leh13}.
\end{proof}

The following result generalises \cite[Theorem 6.1]{Kaw85}, and it will be crucial in the proofs in the following section.

\begin{lem}\label{lem:Kappa=KappaSigma}
Let $(X,\Delta)$ be a klt pair. Then $(X,\Delta)$ has a good model if and only if $\kappa(X,K_X+\Delta)=\kappa_\sigma(X,K_X+\Delta)$.
\end{lem}

\begin{proof}
This is \cite[Theorem 4.3]{GL13}.
\end{proof}

\begin{lem}\label{lem:2}
Let $(X,\Delta)$ and $(X,\Delta')$ be pairs, and assume that there exist $\Q$-divisors $D\geq0$ and $D'\geq0$ such that 
$$K_X+\Delta\sim_\Q D\geq0, \quad K_X+\Delta'\sim_\Q D'\geq0\quad\text{and}\quad\Supp D'=\Supp D.$$
Then 
$$\kappa(X,K_X+\Delta)=\kappa(X',K_{X'}+\Delta')\quad\text{and}\quad\kappa_\sigma(X,K_X+\Delta)=\kappa_\sigma(X',K_{X'}+\Delta')$$
\end{lem}

\begin{proof}
There exist positive rational numbers $t_1$ and $t_2$ such that $t_1D\leq D'\leq t_2D$, hence $\kappa(X,t_1D)\leq\kappa(X,D')\leq\kappa(X,t_2D)$. This implies the first equality, and the second is analogous.
\end{proof}

\subsection{Valuations under finite morphisms}

\begin{dfn}
A geometric valuation $\Gamma$ on a normal variety $X$ is a valuation on the function field $k(X)$ given by the order of vanishing at the generic point of a prime divisor on some proper birational model $f\colon Y\to X$; by abusing notation, we identify $\Gamma$ with the corresponding prime divisor. If $D$ is an $\R$-Cartier divisor on $X$, we use $\mult_\Gamma D$ to denote $\mult_\Gamma f ^*D$. The set $f(\Gamma)$ is the \emph{centre of $\Gamma$ on $X$} and is denoted by $c_X(\Gamma)$.
\end{dfn}

\begin{rem}\label{rem:1}
With notation from the definition, let $R$ be a discrete valuation ring with quotient field $k(X)$ which dominates the local ring $\OO_{X,c_X(\Gamma)} \subseteq k(X)$. Then there exists a morphism $\Spec R\to X$ which sends the generic point of $\Spec R$ to the generic point of $X$, and the closed point of $\Spec R$ to the generic point of $c_X(\Gamma)$, cf.\ \cite[Lemma II.4.4]{Har77}. In particular, this holds if $R=\OO_{Y,\Gamma}$. 
\end{rem}

\begin{rem}\label{rem:2}
Let $X$ be a normal variety and let $(R,m)$ be a discrete valuation ring such that the quotient field of $R$ is $k(X)$. Assume that there is a morphism $\Spec R\to X$ which sends the generic point of $\Spec R$ to the generic point of $X$. Assume that $\trdeg_\C(R/m)=\dim X-1$. Then by a lemma of Zariski \cite[Lemma 2.45]{KM98}, the corresponding valuation is a geometric valuation on $X$.
\end{rem}

We first prove an easy algebraic result that we use in the proof of Proposition \ref{pro:DicrepFinite}. 

\begin{lem}\label{lem:GeomValuation}
Let $k\subseteq K$ be an algebraic extension of fields. Let $(B,m_B)$ be a discrete valuation ring with the quotient field $K$, and let $A=B\cap k$ and $m_A=m_B\cap k$. Then $(A,m_A)$ is a discrete valuation ring with the quotient field $k$ such that the field extension $A/m_A\subseteq B/m_B$ is algebraic.
\end{lem}

\begin{proof}
Let $\nu\colon K\to\Z\cup\{\infty\}$ be the valuation function corresponding to $(B,m_B)$. Then $A=\{a\in k\mid \nu(a)\geq0\}$ and $m_A=\{a\in k\mid \nu(a)>0\}$, and it is immediate that $k$ is the quotient field of $A$. Let $b \in B$ and denote $\overline b=b+m_B\in B/m_B$. Then there is a polynomial
\[
p= T^n + r_{n-1} T^{n-1} + \cdots + r_0\in k[T]
\]  
such that $p(b)=0$, and fix $j \in \{0, \ldots, n-1\}$ such that $\nu(r_j) \leq \nu(r_i)$ for all $i$. If $\nu(r_j) \geq 0$, then $p\in A[T]$ and $\overline b$ is algebraic over $A/m_A$. If $\nu(r_j) < 0$, then $r_j^{-1} \in m_A$ and $\nu(r_j^{-1}r_i) \geq 0$ for all $i$. Therefore, 
$$\overline p=r_j^{-1}p\mod m_A\in (A / m_A)[T]$$ 
is a non-zero polynomial such that $\overline p(\overline b)=0$, which proves the last claim. It remains to show that $m_A\neq\{0\}$. Fix $b\in B$ with $\nu(b)>0$ and let 
\[
p= a_nT^n + a_{n-1} T^{n-1} + \cdots + a_0\in A[T]
\]  
be a polynomial of minimal degree such that $p(b)=0$, so that, in particular, $a_0\neq0$. Then we have 
\[
0<\nu(b)\leq\nu\big(b(a_nb^{n-1} + a_{n-1} b^{n-2} + \cdots +a_1)\big) = \nu(-a_0),
\]  
hence $a_0 \in m_A$.
\end{proof}

\begin{pro}\label{pro:DicrepFinite}
Let $\pi\colon X'\to X$ be a finite morphism of degree $m$ between normal varieties, let $\Delta$ be a $\Q$-divisor on $X$ such that $(X,\Delta)$ is a pair, and let $\Delta'$ be a $\Q$-divisor on $X'$ such that $K_{X'}+\Delta'=\pi^*(K_X+\Delta)$. 
\begin{enumerate}
\item[(i)] For every geometric valuation $E'$ over $X'$ there exists a geometric valuation $E$ over $X$ and an integer $1\leq r\leq m$ such that $\pi(c_{X'}(E'))=c_X(E)$ and 
$$a(E',X',\Delta')+1=r(a(E,X,\Delta)+1).$$
\item[(ii)] For every geometric valuation $E$ over $X$ there exists a geometric valuation $E'$ over $X'$ and an integer $1\leq r\leq m$ such that $\pi(c_{X'}(E'))=c_X(E)$ and 
$$a(E',X',\Delta')+1=r(a(E,X,\Delta)+1).$$
\end{enumerate}
In particular, the pair $(X,\Delta)$ is log canonical (respectively klt) if and only if the pair $(X',\Delta')$ is log canonical (respectively klt).
\end{pro}

\begin{proof}
This is \cite[Proposition 5.20]{KM98}, and in the following we reproduce the proof with more details.

We claim that both in (i) and (ii) there is a commutative diagram
\begin{equation}\label{dia1}
\begin{gathered} 
\xymatrix{ Y^{\prime} \ar[r]^{\pi'} \ar[d]_{f^{\prime}} & Y \ar[d]^f \\
            X^{\prime} \ar[r]^{\pi} & X
 }
\end{gathered}
\end{equation}
where $f$ and $f'$ are birational morphisms, $\pi'$ is finite and there are prime divisors $E \subseteq Y$ and $E^{\prime} \subseteq Y^{\prime}$ such that $\pi'(E^{\prime}) = E$. The claim immediately implies the proposition: indeed, let $r = \mult_{E^{\prime}}(\pi')^*E$. Then locally around the generic point of $E^{\prime}$ we have
\begin{align*}
K_{Y^{\prime}} - (r-1) E ^ {\prime} &= (\pi')^* K_{Y} \sim_\Q (\pi')^*(f^*(K_X + \Delta) + a (E,X,\Delta)\cdot E) \\ 
&= (f')^*(K_X^{\prime} + \Delta^{\prime}) + r\cdot a (E,X,\Delta)\cdot E^{\prime}\\
& \sim_\Q K_{Y^{\prime}} - a (E',X',\Delta')\cdot E^{\prime} + r\cdot a (E,X,\Delta)\cdot E^{\prime},
\end{align*}
hence (i) and (ii) follow.

To see the claim in the case (ii), let $f \colon Y \to X$ be a birational morphism such that $E \subseteq Y$ is a prime divisor, and let $Y'$ be a component of the normalisation of the fibre product $X^{\prime} \times_X Y$ that maps onto $Y$. Then we obtain the diagram \eqref{dia1}, and since $\pi'$ is surjective, there is a prime divisor $E^{\prime} \subseteq Y^{\prime}$ with $\pi'(E^{\prime}) = E$.  

In the case (i), let $(R',m_{R'})$ be the discrete valuation ring corresponding to the valuation $E'$, and let $R=R'\cap k(X)$ and $m_R=m_{R'}\cap k(X)$. Since $k(X)\subseteq k(X')$ is an algebraic extension of fields, $R$ is a discrete valuation ring with quotient field $k(X)$ such that $\trdeg_\C(R/m_R)=\dim X-1$ by Lemma \ref{lem:GeomValuation}. If $E$ is the corresponding discrete valuation, then $E$ is a divisorial valuation by Remark \ref{rem:2}. By Remark \ref{rem:1}, there is a morphism $\rho'\colon\Spec R'\to X'$ which sends the generic point of $\Spec R'$ to the generic point of $X'$, and the closed point of $\Spec R'$ to the generic point $\eta'$ of $c_{X'}(E')$. If $\eta=\pi(\eta')$, then 
$$\OO_{X,\eta}\subseteq\OO_{X',\eta'}\cap k(X)\subseteq R'\cap k(X)=R,$$
hence by Remark \ref{rem:1} there is a morphism $\rho\colon\Spec R\to X$ which sends the generic point of $\Spec R$ to the generic point of $X$, and the closed point of $\Spec R$ to $\eta$. 

Let $f \colon Y \to X$ be a birational morphism such that $E$ is a divisor on $Y$, and denote by $X^{\prime}$ a component of the normalization of the fibre product $X^{\prime} \times_X Y$ that maps onto $Y$, so that we have the diagram \eqref{dia1}. By the valuative criterion of properness, we have the diagram
\[
\xymatrix{ & Y^{\prime} \ar[r]^{\pi'} \ar[d]_{f^{\prime}} & Y \ar[d]^f & \\
\Spec R' \ar[r]^{\quad\rho'} \ar@/_1pc/[rrr]^\iota \ar[ru]^{\theta'}  & X^{\prime} \ar[r]^{\pi} & X & \Spec R \ar[lu]_\theta \ar[l]_{\rho\ \ }
          }
\]
where $\iota\colon\Spec R'\to\Spec R$ is the morphism induced by the inclusion $R\subseteq R'$. Since $f$ is separated, we have $\pi'\circ\theta'=\theta\circ\iota$, and this just says that $E^{\prime}$ is a prime divisor on $Y^{\prime}$ such that $\pi'(c_{Y'}(E^{\prime})) = c_Y(E)$.
\end{proof}

\section{Good models on uniruled pairs}\label{sec:3}

\begin{lem}\label{lem:fibre}
Let $(X,\Delta)$ be a pair, and let $f\colon X\dashrightarrow Y$ be a birational contraction to a normal projective variety such that $K_Y + f_* \Delta$ is $\Q$-Cartier. Then 
$$\kappa_\sigma(X,K_X+\Delta)\leq\kappa_\sigma(Y,K_Y+f_*\Delta).$$ 
\end{lem}

\begin{proof}
Let $(p,q)\colon W\to X\times Y$ be a resolution of the map $f$. Write
$$K_W+\Delta_W\sim_\Q p^*(K_X+\Delta)+E\quad\text{and}\quad K_W+\Delta_W'\sim_\Q q^*(K_Y+f_*\Delta)+E',$$
where $\Delta_W\geq0$ and $E\geq0$ have no common components, and $\Delta_W'\geq0$ and $E'\geq0$ have no common components. Since $f$ is a contraction, the divisor $\Delta_W-\Delta_W'$ is $q$-exceptional, and there are effective $q$-exceptional $\Q$-divisors $E^+$ and $E^-$ such that $\Delta_W-\Delta_W'=E^+-E^-$. Therefore,
$$K_W+\Delta_W+E^-=K_W+\Delta_W'+E^+\sim_\Q q^*(K_Y+f_*\Delta)+E'+E^+,$$
hence $\kappa_\sigma(W,K_W+\Delta_W+E^-)=\kappa_\sigma(Y,K_Y+f_*\Delta)$ by Lemma \ref{lem:properpullbacks}. We conclude since $\kappa_\sigma(X,K_X+\Delta)=\kappa_\sigma(W,K_W+\Delta_W)\leq \kappa_\sigma(W,K_W+\Delta_W+E^-)$ by Lemma \ref{lem:properpullbacks}.
\end{proof}

\begin{dfn}
Let $(X,\Delta)$ be a klt pair. Let $G$ be an effective $\Q$-Cartier $\Q$-divisor such that $K_X+\Delta+G$ is pseudoeffective. Then the \emph{pseudoeffective threshold $\tau(X,\Delta;G)$} is defined as
$$\tau(X,\Delta;G)=\min\{t\in\R\mid K_X+\Delta+tG\text{ is pseudoeffective}\}.$$
\end{dfn}

\begin{thm}\label{thm:tau}
Assume the existence of good models for klt pairs in dimensions at most $n-1$. 

Let $(X,\Delta)$ be a klt pair of dimension $n$. Let $G$ be an effective $\Q$-Cartier $\Q$-divisor such that $(X,\Delta+G)$ is klt and $K_X+\Delta+G$ is pseudoeffective. Assume that $K_X+\Delta$ is not pseudoeffective, i.e.\ that $\tau=\tau(X,\Delta;G)>0$. 

Then $\tau\in\Q$, and there exists a good model of $(X,\Delta+\tau G)$. In particular, 
$$\kappa(X,K_X+\Delta+\tau G)\geq0.$$
\end{thm}

\begin{proof}
We follow closely the proof of \cite[Proposition 8.7, Theorem 8.8]{DHP13}. Fix an ample divisor $A$ on $X$. For any rational number $0\leq x\leq\tau$ let $y_x=\tau(X,\Delta+xG;A)$. Note that $y_\tau=0$ and that $y_x$ is a positive rational number for $0\leq x<\tau$ -- rationality follows from \cite[Corollary 1.1.7]{BCHM}, and positivity from the fact that $K_X+\Delta+xG$ is not pseudoeffective when $x<\tau$.

Let $(x_i)$ be an increasing sequence of non-negative rational numbers such that $\lim\limits_{i\to\infty}x_i=\tau$, and denote $y_i=y_{x_i}$. Fix $i$, let $f_i\colon X\dashrightarrow Y_i$ be the $(K_X+\Delta+x_i G)$-MMP with scaling of $A$, and denote by $\Delta_i$, $G_i$ and $A_i$ the proper transforms of $\Delta$, $G$ and $A$ on $Y_i$. By \cite[Corollary 1.3.3]{BCHM}, there is an extremal contraction $g_i\colon Y_i\to Z_i$ of fibre type such that 
$$K_{Y_i}+\Delta_i+x_iG_i+y_i A_i\equiv_{g_i}0.$$
Let $E_j$ be effective divisors on $Y_i$ whose classes converge to the class of $K_{Y_i} + \Delta_i + \tau G_i$ in $N^1(Y_i)_\R$, and let $C$ be a curve in $Y_i\setminus\bigcup\Supp E_j$ which is contracted by $g_i$. Then 
$$(K_{Y_i} + \Delta_i + \tau G_i) \cdot C \geq 0 \quad \text{and} \quad (K_{Y_i} + \Delta_i + x_i G_i + y_i A_i) \cdot C = 0.$$
Therefore, there exists a rational number $\eta_i\in(x_i,\tau]$ such that $(K_{Y_i}+\Delta_i+\eta_iG_i)\cdot C=0$, hence 
$$K_{Y_i}+\Delta_i+\eta_iG_i\equiv_{g_i}0$$
since all contracted curves are numerically proportional. In particular, if $F_i$ is a general fibre of $g_i$, and $\Delta_{F_i}=\Delta_i|_{F_i}$ and $G_{F_i}=G_i|_{F_i}$, then 
\begin{equation}\label{eq:6}
K_{F_i}+\Delta_{F_i}+\eta_iG_{F_i}\equiv 0.
\end{equation}
Denoting 
$$\tau_i=\max\{t\in\R\mid K_{F_i}+\Delta_{F_i}+tG_{F_i}\text{ is log canonical}\},$$
we have $x_i \leq \tau_i$ since $K_{F_i}+\Delta_{F_i}+x_iG_{F_i}$ is log canonical for every $i$. If $K_{F_i}+\Delta_{F_i}+\tau G_{F_i}$ is not log canonical for infinitely many $i$, then after passing to a subsequence we can assume that $\tau_i < \tau$ for all $i$, and since $x_i \leq \tau_i$ and $\lim x_i = \tau$, we can assume that the sequence $(\tau_i)$ is strictly increasing, which contradicts \cite[Theorem 1.1]{HMX12}. Therefore, $K_{F_i}+\Delta_{F_i}+\tau G_{F_i}$ is log canonical for $i\gg0$, and then \cite[Theorem 1.5]{HMX12} implies that the sequence $(\eta_i)$ is eventually constant, hence $\eta_i=\tau$ for $i\gg0$. In particular, $\tau\in\Q$.

Now, for the rest of the proof fix any such $i\gg0$ for which $\eta_i=\tau$, and let $(p,q)\colon W\to X\times Y_i$ be a resolution of the map $f_i$. 
$$\xymatrix{ & W \ar[ld]_{p} \ar[dr]^{q} & & \\
            X \ar@{-->}[rr]^{f_i} & & Y_i \ar[r]^{g_i} & Z_i
 }
$$
We may write
$$K_W+\Delta_W\sim_\Q p^*(K_X+\Delta+\tau G)+E,$$
where $\Delta_W$ and $E$ are effective $\Q$-divisors without common components. We want to prove that $(X, \Delta + \tau G)$ has a good minimal model, hence by Lemmas \ref{lem:properpullbacks} and \ref{lem:Kappa=KappaSigma}, it is enough to show that
\begin{equation}\label{eq:8}
\kappa(W,K_W+\Delta_W)=\kappa_\sigma(W,K_W+\Delta_W).
\end{equation}
If we denote $F_W=q^{-1}(F_i)\subseteq W$, then $q_*(K_{F_W}+\Delta_W|_{F_W})=K_{F_i}+\Delta_{F_i}+\tau G_{F_i}$, hence by Lemma \ref{lem:fibre} and by \eqref{eq:6},
\begin{equation}\label{eq:7}
\kappa_\sigma(F_W,K_{F_W}+\Delta_W|_{F_W})\leq\kappa_\sigma(F_i,K_{F_i}+\Delta_{F_i}+\tau G_{F_i})=0.
\end{equation}
When $\dim Z_i=0$, then $F_W=W$ and \eqref{eq:7} implies \eqref{eq:8} by \cite[Corollary V.4.9]{Nak04}. 

When $\dim Z_i>0$, then by Theorem \ref{thm:birkar} and by \cite[Theorem 1.1]{Fuj11} there exists a good model $(W,\Delta_W)\dashrightarrow (W_{\min},\Delta_{\min})$ of $(W,\Delta_W)$ over $Z_i$. Let $\varphi\colon W_{\min}\to W_\can$ be the corresponding fibration to the canonical model of $(W,\Delta_W)$ over $Z_i$. Since $K_W+\Delta_W$ is not big over $Z_i$ by \eqref{eq:7}, we have $\dim W_\can<\dim X$. By \cite[Theorem 0.2]{Amb05a}, there exists a divisor $\Delta_\can$ on $W_\can$ such that the pair $(W_\can,\Delta_\can)$ is klt and 
$$K_{W_{\min}}+\Delta_{\min}\sim_\Q\varphi^*(K_{W_\can}+\Delta_\can).$$
Since we assume the existence of good models for klt pairs in dimensions at most $n-1$, we have $\kappa(W_\can,K_{W_\can}+\Delta_\can)=\kappa_\sigma(W_\can,K_{W_\can}+\Delta_\can)$ by Lemma \ref{lem:Kappa=KappaSigma}, and hence \eqref{eq:8} holds by Lemma \ref{lem:properpullbacks}, which concludes the proof. 
\end{proof}

\begin{rem}\label{rem:DHP}
Let $(X,\Delta)$ be a uniruled klt pair such that $K_X$ is not pseudoeffective and $K_X+\Delta$ is pseudoeffective. A natural strategy to construct a good model of $(X,\Delta)$ is to run a $(K_X+\tau\Delta)$-MMP, where $\tau=\tau(X,0;\Delta)$, and which we know terminates with a good model $(Y,\Delta_Y)$ by Theorem \ref{thm:tau}. The main problem is that this MMP does not preserve sections of $K_X+\Delta$. An instructive example is when $K_X\sim_\Q{-}\tau\Delta$, where $\Delta$ is nef and not big, and for instance $\rho(X)=2$. Then one might want to run the $(K_X+(\tau-\varepsilon)\Delta)$-MMP with scaling of an ample divisor $A$, where $0<\varepsilon\ll1$. If $\Nef(X)\neq\Effb(X)$, then this MMP ends up with a model on which the proper transform of $K_X+\Delta$ is ample, regardless of the Kodaira dimension of $K_X+\Delta$.
\end{rem}

\begin{thm}\label{thm:reduction}
Assume the existence of good models for klt pairs in dimensions at most $n-1$, and the existence of good models for log smooth klt pairs $(X,\Delta)$ in dimension $n$ such that $|K_X|\neq\emptyset$. 

Let $(X,\Delta)$ be a log smooth log canonical pair of dimension $n$ and assume that there exists a $\Q$-divisor $D\geq0$ such that $K_X+\Delta\sim_\Q D$ and $\Supp\Delta=\Supp D$. Then 
$$\kappa(X,K_X+\Delta)=\kappa_\sigma(X,K_X+\Delta).$$
\end{thm}

\begin{proof}
Replacing $\Delta$ by $\lceil\Delta\rceil$, by Lemma \ref{lem:2} we may assume that the divisor $\Delta$ is reduced. In the course of the proof, we construct a tower of proper maps
$$(T,\Delta_T)\stackrel{\mu}{\lto} (W,\Delta_W)\stackrel{g}{\lto} (Z,\Delta_Z)\stackrel{f}{\lto} (X',\Delta_{X'})\stackrel{\pi}{\lto} (X,\Delta),$$
where $\pi$ and $\mu$ are finite, and $f$ and $g$ are birational, such that for each $\mathcal X\in\{T,W,Z,X'\}$ we have
$$\kappa(\mathcal X,K_{\mathcal X}+\Delta_{\mathcal X})=\kappa(X,K_X+\Delta)\quad\text{and}\quad\kappa_\sigma(\mathcal X,K_{\mathcal X}+\Delta_{\mathcal X})=\kappa_\sigma(X,K_X+\Delta).$$
The pair $(T,\Delta_T)$ will be log smooth with $|K_T|\neq\emptyset$ which allows us to conclude.

Let $m$ be the smallest positive integer such that $m(K_X+\Delta)\sim mD$, and denote $G=mD$. Let $\pi\colon X'\to X$ be the normalisation of the corresponding $m$-fold cyclic covering ramified along $G$. Note that $X'$ is irreducible by \cite[Lemma 3.15(a)]{EV92} since $m$ is minimal. Then there exists an effective Cartier divisor $G'$ on $X'$ such that 
$$\pi^*G=mG'\quad\text{and}\quad\pi^*(K_X+\Delta)\sim G',$$
and let $\Delta'=(G')_\textrm{red}$. By the Hurwitz formula, we have
$$K_{X'}+\Delta'=\pi^*(K_X+\Delta),$$
and the pair $(X',\Delta')$ is log canonical by Proposition \ref{pro:DicrepFinite}. By Theorem \ref{thm:dltblowup}, there exists a dlt model $f\colon (Z,\Delta_Z)\to X'$ of $(X',\Delta')$, and we have
$$\kappa(X,K_X+\Delta)=\kappa(Z,K_Z+\Delta_Z)\quad\text{and}\quad\kappa_\sigma(X,K_X+\Delta)=\kappa_\sigma(Z,K_Z+\Delta_Z)$$
by Lemma \ref{lem:properpullbacks}. Denote $G_Z=f^*G'$. We claim:
\begin{enumerate}
\item[(i)] for every geometric valuation $E'$ over $Z$ we have $a(E',Z,\Delta_Z)\in\Z$,
\item[(ii)] $\Supp\Delta_Z\subseteq\Supp G_Z$.
\end{enumerate}

To prove the claim, let $E'$ be a geometric valuation over $Z$. Then by Proposition \ref{pro:DicrepFinite}, there exists a geometric valuation $E$ over $X$ and an integer $1\leq r\leq m$ such that
\begin{equation}\label{eq:1}
a(E',Z,\Delta_Z)+1=a(E',X',\Delta')+1=r(a(E,X,\Delta)+1),
\end{equation}
where the first equality holds because $K_Z+\Delta_Z\sim_\Q f^*(K_{X'}+\Delta')$. Since $(X,\Delta)$ is log smooth and $\Delta$ is reduced, we have $a(E,X,\Delta)\in\Z$, which together with \eqref{eq:1} implies (i). 

To show (ii), let $S'$ be a component of $\Delta_Z$. Then $a(S',X',\Delta')=-1$ by the definition of dlt models. By Proposition \ref{pro:DicrepFinite}, there exists a geometric valuation $S$ over $X$ and an integer $1\leq r\leq m$ such that $\pi(c_{X'}(S'))=c_X(S)$ and
$$a(S',X',\Delta')+1=r(a(S,X,\Delta)+1).$$
This implies $a(S,X,\Delta)=-1$, thus $c_X(S)\subseteq\Supp\Delta$ because $(X,\Delta)$ is log smooth. From here we obtain $c_{X'}(S')\subseteq\pi^{-1}(\Supp\Delta)=\Supp G'$, and in particular $S'\subseteq\Supp G_Z$.

Now, if $g\colon W\to Z$ is a log resolution of the pair $(Z,\Delta_Z)$, by (i) above we may write
$$K_W+\Delta_W\sim_\Q g^*(K_Z+\Delta_Z)+E_W\sim_\Q g^*G_Z+E_W,$$
where $\Delta_W$ and $E_W$ are effective integral divisors with no common components. Then
$$\kappa(X,K_X+\Delta)=\kappa(W,K_W+\Delta_W)\quad\text{and}\quad\kappa_\sigma(X,K_X+\Delta)=\kappa_\sigma(W,K_W+\Delta_W)$$
by Lemma \ref{lem:properpullbacks}, and the divisor $G_W=g^*G_Z+E_W-\Delta_W$ is Cartier. We have 
$$K_W\sim_\Q G_W,$$
and we claim that $G_W\geq0$. Indeed, if $S$ is a component of $\Delta_W$, then $a(S,Z,\Delta_Z)=-1$, and hence by the definition of dlt singularities we have $c_Z(S)\subseteq\Supp\Delta_Z$. By (ii) above, this implies $\mult_Sg^*G_Z\geq1=\mult_S\Delta_W$, hence the claim follows. 

Now, consider the klt pair $(K_W,\frac12\Delta_W)$. Since $K_W+\frac12\Delta_W\sim_\Q G_W+\frac12\Delta_W$, $K_W+\Delta_W\sim_\Q G_W+\Delta_W$ and $\Supp(G_W+\frac12\Delta_W)=\Supp(G_W+\Delta_W)$, by Lemma \ref{lem:2} we have
$$\textstyle\kappa(X,K_X+\Delta)=\kappa(W,K_W+\frac12\Delta_W)\quad\text{and}\quad\kappa_\sigma(X,K_X+\Delta)=\kappa_\sigma(W,K_W+\frac12\Delta_W).$$
Let $k$ be the smallest positive integer such that $k(K_W-G_W)\sim0$, and let $\mu\colon T\to W$ be the corresponding $k$-fold \'etale covering. Then 
$$K_T=\mu^*K_W\sim\mu^*G_W,$$
and setting $\Delta_T=\mu^*(\frac12\Delta_W)$, the pair $(K_T,\Delta_T)$ is klt by Proposition \ref{pro:DicrepFinite}. We have
$$\kappa(X,K_X+\Delta)=\kappa(T,K_T+\Delta_T)\quad\text{and}\quad\kappa_\sigma(X,K_X+\Delta)=\kappa_\sigma(T,K_T+\Delta_T)$$
by Lemma \ref{lem:2}, hence $\kappa(X,K_X+\Delta)=\kappa_\sigma(X,K_X+\Delta)$ by our assumptions and by Lemma \ref{lem:Kappa=KappaSigma}. 
\end{proof}

\begin{rem}
With the notation from the proof of Theorem \ref{thm:reduction}, one can show that the variety $Z$ has canonical singularities, so that $Z$ is not uniruled by Theorem \ref{thm:BDPP}, without passing to a log resolution.
\end{rem}

\begin{rem}\label{rem:toroidal}
In the proof of Theorem \ref{thm:reduction}, $X'\setminus\Delta'\subseteq X'$ is a toroidal embedding since the pair $(X,\Delta)$ is log smooth \cite[Lemma 1.1]{Ara14}, i.e.\ it is locally analytically on $X'$ isomorphic to an embedding of a torus into a toric variety. By \cite[Theorem 0.2]{AW97}, there exists a toroidal resolution $h\colon (U,\Delta_U)\to(X',\Delta')$ and then $K_U+\Delta_U=h^*(K_{X'}+\Delta')$: indeed, locally in the analytic category both sides of this equation are trivial, which implies that all discrepancies are zero. This is all implicit already in \cite{KKMSD}. The pair $(U,\Delta_U)$ is log smooth, and as in the proof of Theorem \ref{thm:reduction}, one shows that $K_U$ is linearly equivalent to an effective Cartier divisor. Therefore, if one prefers toroidal embeddings, one can avoid the use of dlt models; however, compare to \cite[Section 5]{dFKX12}.
\end{rem}

Finally we can prove our main results.

\begin{proof}[Proof of Theorem \ref{thm:main}]
Let $(X,\Delta)$ be a uniruled klt pair. By replacing $(X,\Delta)$ by its terminal model, cf.~Theorem \ref{thm:dltblowup}(a), we may assume that the pair $(X,\Delta)$ is terminal, and thus that $K_X$ is not pseudoeffective by Theorem \ref{thm:BDPP}. Let $\tau=\tau(X,0;\Delta)=\min\{t\in\R\mid K_X+t\Delta\text{ is pseudoeffective}\}$. Since $K_X$ is not pseudoeffective and $K_X+\Delta$ is pseudoeffective, we have $0<\tau\leq1$. If $\tau=1$, then we conclude by Theorem \ref{thm:tau}. 

Therefore, we may assume that $\tau<1$, and hence by Theorem \ref{thm:tau} there exists a $\Q$-divisor $D_\tau\geq0$ such that $K_X+\tau\Delta\sim_\Q D_\tau$. This yields 
$$K_X + \Delta \sim_\Q D \geq 0 , \quad \text{where} \quad D =D _\tau + (1-\tau )\Delta.$$
In particular, $\Supp\Delta\subseteq\Supp D$. Let $f\colon Y\to X$ be a log resolution of the pair $(X,D)$. Then we may write
$$K_Y+\Gamma\sim_\Q f^*(K_X+\Delta)+E,$$
where $\Gamma$ and $E$ are effective $\Q$-divisors with no common components, and $\Gamma=f^{-1}_*\Delta$ since $(X,\Delta)$ is a terminal pair. In particular, if we denote $D_Y=f^*D+E$, then $K_Y+\Gamma\sim_\Q D_Y$ and $\Supp\Gamma\subseteq\Supp D_Y$. We have
$$\kappa(X,K_X+\Delta)=\kappa(Y,K_Y+\Gamma)\quad\text{and}\quad \kappa_\sigma(X,K_X+\Delta)=\kappa_\sigma(Y,K_Y+\Gamma)$$
by Lemma \ref{lem:properpullbacks}, hence by replacing $(X,\Delta)$ by $(Y,\Gamma)$ and $D$ by $D_Y$, we may assume that $(X,D)$ is a log smooth pair. Finally, by replacing $\Delta$ by $\Delta+\varepsilon D$ for $0<\varepsilon\ll1$, we may further assume that $\Supp\Delta=\Supp D$. We conclude by Theorem \ref{thm:reduction} and by Lemma \ref{lem:Kappa=KappaSigma}.
\end{proof}

\begin{proof}[Proof of Theorem \ref{thm:main2}]
Let $(X,\Delta)$ be a uniruled klt pair. As in the proofs of Theorems \ref{thm:main} and \ref{thm:reduction}, there exists a log smooth klt pair $(T,\Delta_T)$ such that $|K_T|\neq\emptyset$ and 
$$\kappa(X,K_X+\Delta)=\kappa(T,K_T+\Delta_T)\geq0\quad\text{and}\quad \kappa_\sigma(X,K_X+\Delta)=\kappa_\sigma(T,K_T+\Delta_T).$$
In particular, $T$ is not uniruled by Theorem \ref{thm:BDPP}. By Theorem \ref{thm:birkar}, there exists a log terminal model $(T,\Delta_T)\dashrightarrow (T',\Delta_{T'})$ of $(T,\Delta_T)$, hence 
$$\kappa(T',K_{T'}+\Delta_{T'})=\kappa_\sigma(T',K_{T'}+\Delta_{T'})$$
since we assume the abundance conjecture for non-uniruled pairs. We conclude by Lemmas \ref{lem:properpullbacks} and \ref{lem:Kappa=KappaSigma}.
\end{proof}

\begin{proof}[Proof of Theorem \ref{thm:main3}]
Immediate from Theorem \ref{thm:main}.
\end{proof}

\begin{rem}\label{rem:goodmodels}
Assume that for every smooth variety of dimension $n$ with $K_X$ pseudoeffective we have $\kappa(X,K_X)\geq0$. Then the previous proofs show that if good models exist for log smooth klt pairs $(X,\Delta)$ of dimension $n$ such that the linear system $|K_X|$ is not empty, then good models exist for klt pairs in dimension $n$. 

Indeed, by Theorem \ref{thm:main3} we only have to show that the assumptions imply the existence of good models for non-uniruled klt pairs in dimension $n$. Fix such a pair $(X,\Delta)$, and note that we may assume that the pair is terminal by Theorem \ref{thm:dltblowup}. Then $\kappa(X,K_X)\geq0$ by our assumption, hence there exists an effective divisor $D'$ such that $K_X\sim_\Q D'$. In particular, by denoting $D=D'+\Delta$ we have $K_X+\Delta\sim_\Q D$ and $\Supp \Delta\subseteq\Supp D$. As in the proof of Theorem \ref{thm:main}, by passing to a log resolution, we may assume that $(X,D)$ is log smooth. By replacing $\Delta$ by $\Delta+\varepsilon D$ for $0<\varepsilon\ll1$, we may further assume that $\Supp\Delta=\Supp D$, and we conclude by Theorem \ref{thm:reduction} and by Lemma \ref{lem:Kappa=KappaSigma}. 
\end{rem}

This leads to the following result.

\begin{lem}\label{lem:IndexOne}
Let $(X,\Delta)$ be a $\Q$-factorial terminal pair such that $\kappa(X,K_X)\geq0$. Then there exists a generically finite morphism $f\colon Y\to X$ from a smooth variety $Y$ and an effective $\Q$-divisor $\Gamma$ on $Y$ with simple normal crossings support such that the pair $(Y,\Gamma)$ is klt, $|K_Y|\neq\emptyset$ and
$$\kappa(X,K_X+\Delta)=\kappa(Y,K_Y+\Gamma)\quad\text{and}\quad \kappa_\sigma(X,K_X+\Delta)=\kappa_\sigma(Y,K_Y+\Gamma).$$
If $\Delta=0$, we may additionally assume that $\Gamma=0$.
\end{lem}

\begin{proof}
The first claim follows from the proof of Theorem \ref{thm:main}. When $\Delta=0$, as in Remark \ref{rem:goodmodels} we may assume that $X$ is smooth and that there exists a $\Q$-divisor $D\geq0$ with simple normal crossings support such that $K_X\sim_\Q D'$. Setting $\Delta_X=\varepsilon D$ and $D=D'+\Delta_X$ for a rational number $0<\varepsilon\ll1$, we have $K_X+\Delta_X\sim_\Q D$ and $0<\mult_E\Delta_X<\mult_E D$ for every component $E$ of $D$. Then with notation from the proof of Theorem \ref{thm:reduction}, we obtain a generically finite map $(W,\Delta_W)\to(X,\Delta_X)$ such that the pair $(W,\Delta_W)$ is log smooth,
$$\kappa(W,K_W+\Delta_W)=\kappa(X,K_X+\Delta_X)\quad\text{and}\quad\kappa_\sigma(W,K_W+\Delta_W)=\kappa_\sigma(X,K_X+\Delta_X),$$
and $K_W\sim_\Q G_W$ for some Cartier divisor $G_W$ such that -- crucially -- $\Supp G_W=\Supp(G_W+\Delta_W)$. In particular, by Lemma \ref{lem:2} this implies
$$\kappa(W,K_W)=\kappa(X,K_X)\quad\text{and}\quad\kappa_\sigma(W,K_W)=\kappa_\sigma(X,K_X).$$
Finally, one more \'etale cover allows to conclude as in the proof of Theorem \ref{thm:reduction}.
\end{proof}

\bibliographystyle{amsalpha}

\bibliography{biblio}

\end{document}